\documentclass[reqno, 12pt,a4paper]{amsart}

\usepackage[T2A]{fontenc}
\usepackage[utf8]{inputenc}
\usepackage[english]{babel}

\usepackage{amsfonts,amssymb,amsthm,mathrsfs,amscd}
\usepackage{graphicx}

\theoremstyle{plain}
\newtheorem{theorem}{Theorem}
\newtheorem{lemma}{Lemma}
\newtheorem{corollary}{Corollary}
\newtheorem{proposition}{Proposition}

\theoremstyle{definition}

\newtheorem{remark}{Remark}

\begin{document}

\title{Spectra of Dirichlet Laplacian in 3-dimensional polyhedral layers}

\author[Bakharev F. L.]{Bakharev F. L.${}^\dagger$}

\thanks{\textsc{${}^\dagger$}{Chebyshev Laboratory, St. Petersburg State University, 14th Line V.O., 29, Saint Petersburg 199178 Russia}}

\email{fbakharev@yandex.ru}

\author[Matveenko S. G.]{Matveenko S. G.${}^\dagger$}

\email{matveis239@gmail.com}

\thanks{{\it E-mail addresses:} \texttt{fbakharev@yandex.ru, matveis239@gmail.com}}


\begin{abstract}
The structure of the spectrum of the three-dimensional Dirichlet Laplacian in the 3D polyhedral layer of fixed width is studied.
It appears that the essential spectrum is defined by the smallest dihedral angle that forms the boundary of the layer while the discrete spectrum is always finite. An example of a layer with the empty discrete spectrum is constructed. The spectrum is proved to be nonempty in regular polyhedral layer.

\end{abstract}

\thanks{The work is supported by the Russian Science Foundation grant 19-71-30002}

\maketitle

\section{Introduction}

The study of the Dirichlet spectral problem for the Laplace operator in various unbounded domains has attracted the attention of researchers during last decades. In recent years, many papers on spectra of quantum waveguides (regions with cylindrical outlets to infinity and Dirichlet conditions at the boundary) have appeared. A slightly smaller number of works discuss the so-called Dirichlet layers and their junctions, which structurally may be much more complicated than waveguides. Nevertheless, spectra of problems in the layers are similar to spectra of problems in waveguides. This is due to the fact that in layers, as well as in waveguides, on the one hand, a countable number of identical non-intersecting cubes can be inscribed, and on the other hand, one cannot inscribe an arbitrarily large cube. This leads to the fact that the corresponding problems usually have an essential spectrum $\sigma_{ess}$ occupying the ray on the real axis with a positive cut-off point. Below its threshold there may exist a discrete spectrum $\sigma_d$ corresponding to trapped waves. The appearance of trapped waves is often a consequence of the structure of the domain contained in a sufficiently large ball (the shape of the junction), while the characteristics of the essential spectrum are determined by the structure of the waveguide or layer at infinity.

Several works focused on Dirichlet layers are worth mentioning. In \cite{DuExKr01} the authors study the spectra of Dirichlet layers with constant width built around a special type of two-dimensional smooth surfaces in $\mathbb{R}^3$, assuming that the surface ``flattens out'' at infinity, meaning its curvature approaches zero. The authors provide sufficient conditions for trapped wave emergence in such layers. In \cite{CaExKr04} the authors expand on the ideas from \cite{DuExKr01}, dropping some assumptions, and consider more complex topological structures and a wider class of layers. Nonetheless, the flattening out of the generating surface at infinity remains a key assumption. Note also \cite{LiLu07} and \cite{LuRo12} which develop and generalize the results from the aforementioned two articles, in particular, to the multi-dimensional case. In \cite{ExTa10}, the authors allow for a curvature explosion at one of the surface points, considering a conical layer. This leads to a surprising result: an infinite number of eigenvalues below the threshold appears.

This work studies the spectral properties of the Dirichlet problems in constant width layers constructed from polyhedral three-dimensional angles. The analog in two-dimensional plane are the ``${\sf L}$-shaped'' waveguides, the spectra of which have been well studied (see, for example, works \cite{ExSeSt89, Na11, DaRa12, NaSh14} and monograph \cite{ExKo15}). In the first known work \cite{DaLaOu18} the question of the structure of the essential spectrum of a Fichera layer (a layer constructed from a trihedral angle with all flat angles equal to $90^\circ$) is studied and the finiteness of the number of eigenvalues below the threshold is established. The result is refined in work~\cite{BaAI20} where the existence of trapped waves is proven, and the same result is established in higher dimensions. The uniqueness of the eigenvalue below the threshold of the continuous spectrum in this situation remains an open question, which has only been verified numerically.

This work extends the research in \cite{BaAI20} by considering the Dirichlet problem for the Laplacian operator in a layer $\Pi$ built on a three-dimensional polyhedral angle. The assumption that a sphere can be placed in the angle to touch all its faces is imposed. The structure of the essential spectrum is established, namely, it is shown that it covers the ray with a cut-off point which coincides with the first eigenvalue $\Lambda_\dagger$ of the Dirichlet problem in an {\sf L}-shaped waveguide with an opening angle equal to the minimum of the dihedral angles of $\Pi$. 
As the essential spectrum consists of non-isolated points, it coincides with the continuous spectrum of the problem. It should be noted that the polyhedral layer may have a rich symmetry group and hence it may occur eigenvalues embedded in continuous spectrum (see, for instance, \cite{EvLeVa}). Thus the fine structure of the essential spectrum remains an open question. 
It is then established that for regular polyhedral angles the corresponding layers necessarily have a non-empty discrete spectrum. However, a general statement of this sort cannot be made --- an example of a trihedral layer with an empty discrete spectrum is provided. Note that in \cite{LiLu07} it is proven that for sufficiently thin constant thickness layers built on the graphs of smooth strictly convex functions in $\mathbb{R}^3$ the discrete spectrum is non-empty when the layers are flattened at infinity.

The work has the following structure. Section \ref{sec-2} contains a description of the domain, the statement of the Dirichlet problem and preliminary information about the spectrum. In Section \ref{sec-3} known results on {\sf L}-shaped waveguides are collected. In Section~\ref{sec-4} we study the structure of the spectrum, specifically, we establish the finiteness of the discrete spectrum and prove that the essential spectrum is the ray $[\Lambda_\dagger,+\infty)$. Note that this proof is slightly different from the ones proposed in the works \cite{DaLaOu18} and \cite{BaAI20} and, in our opinion, is more elementary.

In Section \ref{sec-6} an example of a trihedral layer is given, for which the discrete spectrum is empty. In Section \ref{sec-7} we provide a proof of the non-emptiness of the discrete spectrum for regular polyhedral layers. This proof generally repeats the proof from the work \cite{BaAI20} and is given for the sake of completeness. Section~\ref{sec-8} contains possible generalizations and questions that remain unanswered.

\section{Problem statement}
\label{sec-2}
\subsection{Domain geometry}
\label{sub-sec-2.1}

We consider a convex $n$-gon $P_1'P_2'\ldots P_n'$ ($n\geq 3$) in $\mathbb{R}^3$, assuming that its plain doesn't contain an origin $O$. A convex hull of rays  $\ell'_j=O'P_j'$ forms a polyhedral angle with $n$ faces 
$$\Upsilon' = \mathop{\rm conv}\nolimits\{ \ell'_j \colon 1\leq j\leq n\}.$$ 
Its boundary $\Gamma'=\partial \Upsilon'$ consists of vertex angles $\Gamma_j'=\mathop{\rm conv}\nolimits\{\ell'_j, \ell'_{j+1}\}$, which are equal to $\alpha_j=\angle P_{j}'O'P_{j+1}'$. The value of the inner dihedral angle with the edge $\ell'_j$ we define by~$\beta_j$.
Here and below we assume that indices are equivalent to each other with respect to congruence modulo $n$, i.e. $n+1=1$, $n+2=2$, \ldots.  

\begin{figure}[ht!]
    \centering
    \includegraphics[width=0.45\textwidth]{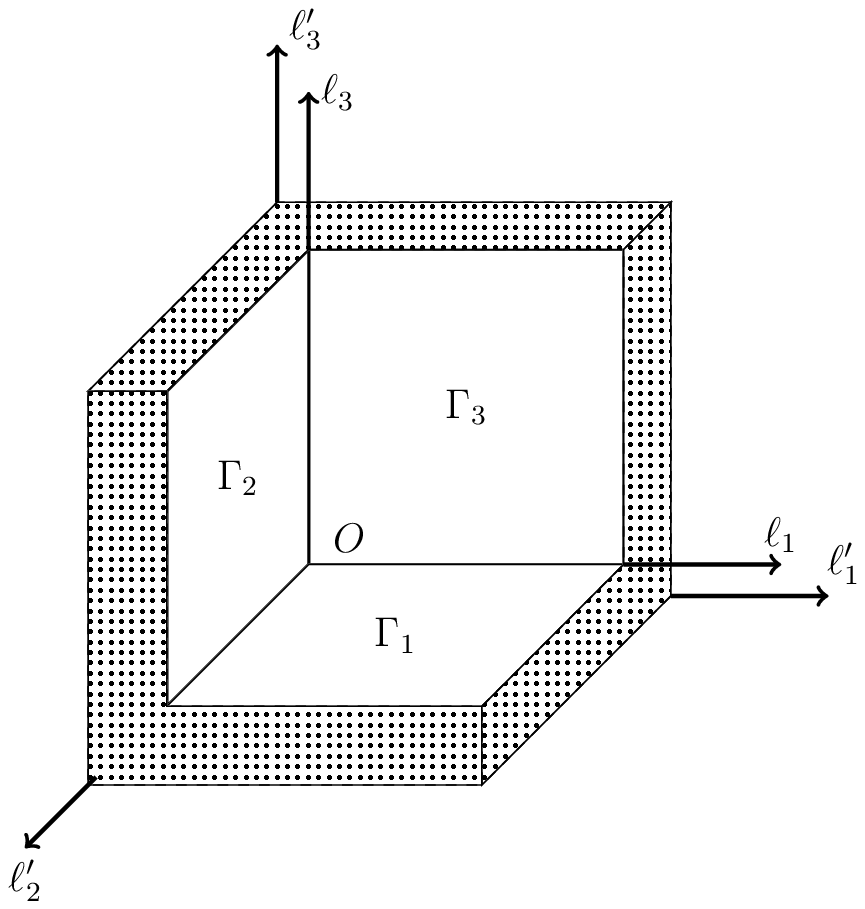}  \qquad
    \includegraphics[width=0.45\textwidth]{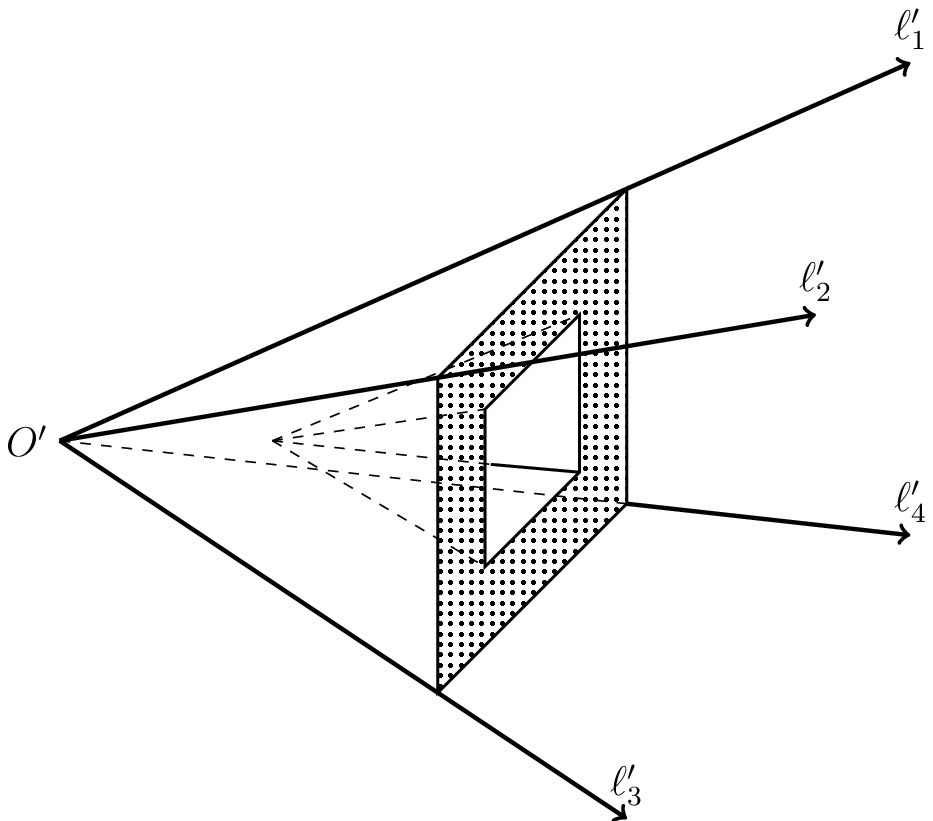}  
    \caption{Layers built on regular trihedral and tetrahedral angles}
    \label{fig-01}
\end{figure}

The polyhedral layer of unit width is defined by the formula 
$$
\Pi=\{x_\dagger\in \Upsilon' \colon \mathop{\rm dist}\nolimits(x, \Gamma')\in (0,1)\}.
$$ 
The boundary $\partial \Pi$ consists of two polyhedral surfaces: the outer one coincides with the~ $\Gamma'$ and the inner one we denote by $\Gamma$. If a ball can be inscribed into the angle $\Upsilon'$, then $\Gamma'$ and $\Gamma$ can be matched by shifting. Otherwise, the inner polyhedral surface $\Gamma$ may be more complex, but in this work we consider only the first of these situations, which occurs, for instance, in the case of any trihedral angle, as well as in the case of a regular polyhedral angle (an angle with equal vertex angles and equal dihedral angles). For elements of $\Gamma$, corresponding to elements $\Gamma'$, we use the same symbols, but removing the stroke.

We associate a rectangular Cartesian coordinate system $(x_j, y_j, z_j)$ with each dihedral angle of the polyhedral angle $\Upsilon$ defined by its origin at point $O$ and an orthonormal basis $({\bf e}_j^1, {\bf e}_j^2, {\bf e}_j^3)$, where the latter vector of the triple is directed along the ray $\ell_j$, the first of the triple lies within one of the faces, and the second is directed into the half-space containing the angle.

Below we need one natural decomposition of $\Pi$. For each vertex angle $\Gamma_j$, we will draw a bisector and construct a plane $\pi_j$ passing through this bisector, perpendicular to the plane $\Gamma_j$. The region cut out from the layer $\Pi$ by two such planes $\pi_{j-1}$ and $\pi_j$ is denoted by $\varpi_j$. So, the closure of layer $\Pi$ is represented as a union
\begin{equation} \label{partition}
    \overline{\Pi} = \bigcup\limits_{j=1}^n \overline{\varpi_j}.
\end{equation}

\begin{figure}[ht!]
    \centering
    \includegraphics[width=0.4\textwidth]{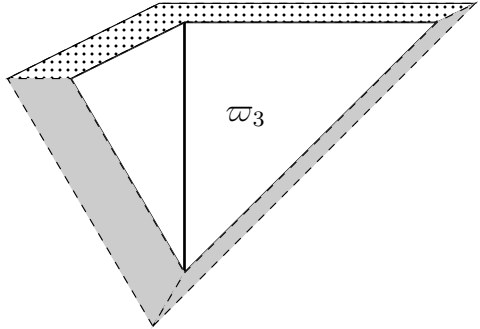}  
    \caption{The set $\varpi_3$ for the trihedral layer $\Pi$ on  Fig. \ref{fig-01}}
\end{figure}

Let us note that the sections of $\varpi_j$, which are orthogonal to the axes $Oz_j$ for $z_j > 0$, are truncated broken waveguides of unit width with inner angle magnitude equals $\beta_j$. The  coordinate system $z_j^\perp = (x_j, y_j)$, whose origin coincides with the inner vertex, corresponds to this section. 

\subsection{Spectral problem and preliminary information about the spectrum}
In this work we study the spectral problem of the Laplace operator
\begin{equation}
    \label{problem}
    -\Delta u(x)=\Lambda u(x), \quad x\in \Pi; \quad u(x)=0,\quad x\in \partial\Pi
\end{equation}
in various domains $\Pi$. The problem \eqref{problem} admits a variational formulation
$$
\mathfrak{a}_\Pi[u,v]:=(\nabla u, \nabla v)_\Pi=\Lambda (u,v)_\Pi \quad \forall v\in \lefteqn{\overset\circ{\hphantom{H'}\vphantom{\prime}}}H\vphantom{H}^1(\Pi),
$$
where $(u,v)_{\Pi}$ denotes the standard scalar product in $L_2(\Pi)$. The closed positive definite sesquilinear form $\mathfrak{a}_\Pi$ on the Sobolev space of functions with zero trace on the boundary $\partial \Pi$ in $\lefteqn{\overset\circ{\hphantom{H'}\vphantom{\prime}}}H\vphantom{H}^1(\Pi)$ generates a positive self-adjoint operator $\mathfrak{A}_\Pi$. The structure of the spectrum of this operator in the domains from Sec. \ref{sub-sec-2.1} is the main subject of study of this work.

As will be shown in Theorem \ref{ess_sp_th}, the essential spectrum of the operator $\sigma_{ess}(\mathfrak{A}_\Pi)$ occupies the ray $[\Lambda_\dagger, +\infty)$ with a cut-off $\Lambda_\dagger=\Lambda_\dagger(\Pi)$ which coincides with the first eigenvalue in the ${\sf L}$-shaped waveguide with the inner vertex angle equals to the smallest of the dihedral angles $\beta_j$.

The eigenvalues of the operator $\mathfrak{A}_\Pi$ that are located below the threshold $\Lambda_\dagger$ are elements of the discrete spectrum $\sigma_d(\mathfrak{A}_\Pi)$ and can be found using the max-min principle (see \cite[Theorem 10.2.2]{BiSo80}).
\[
\Lambda_j(\Pi)\ \ =\sup\limits_{\mathop{\rm codim}\nolimits E=j-1}\ \inf\limits_{u\in E\setminus{\{\vec{0}\}}}\ 
\frac{\|\nabla u;L_2(\Pi)\|^2}{\|u;L_2(\Pi)\|^2},
\] 
where the supremum is taken over all subspaces $E$ of codimension $j-1$ in the space $\lefteqn{\overset\circ{\hphantom{H'}\vphantom{\prime}}}H\vphantom{H}^1(\Pi)$. 
In particular, the existence of a function $w\in\lefteqn{\overset\circ{\hphantom{H'}\vphantom{\prime}}}H\vphantom{H}^1(\Pi)$ such that
 \[
 \|\nabla w;L_2(\Pi)\|^2<\Lambda_\dagger\|w;L_2(\Pi)\|^2,
 \]
is equivalent to the existence of points of the discrete spectrum below the threshold. As for the first eigenvalue $\Lambda_1$ of the Dirichlet problem if it exists, it is known to be simple, and the corresponding eigenfunction can be chosen positive.

Theorem \ref{Finitness_of_discrete_sp} states that the discrete spectrum, if exists, is finite and lies below the threshold. For a regular polyhedral angular layer the discrete spectrum is always non-empty (see theorem \ref{ExistDiscrSp}). Theorem \ref{abs_of_ds} gives an example of an polyhedral angular layer with an empty discrete spectrum.

\section{Spectra of ${\sf L}$-shaped waveguides}
\label{sec-3}

Below we will need some properties of planar waveguides formed by a broken at an angle $\theta\in (0,\pi)$ strip of unit width (see Fig.~\ref{fig-05}). Let us define them in a similar way to the section \ref{sub-sec-2.1}. Consider two rays $\ell_1'$ and $\ell_2'$ on a plane with the vertex  $O'$ and angle $\theta$ between them. Their union $\gamma'=\ell_1'\cup \ell_2'$ forms an exterior part of the waveguide boundary. Define $\omega(\theta)$ as a set of all points inside the angle located at a distance less than one from $\gamma'$. The boundary of such waveguide is a disjunctive union $\partial \omega(\theta) = \gamma(\theta)\cup \gamma'(\theta)$ of the ``interior'' and ``exterior'' parts. The vertex of the interior angle we denote by $O$.

\begin{figure}[ht!]
    \centering
    \includegraphics[width=0.6\textwidth]{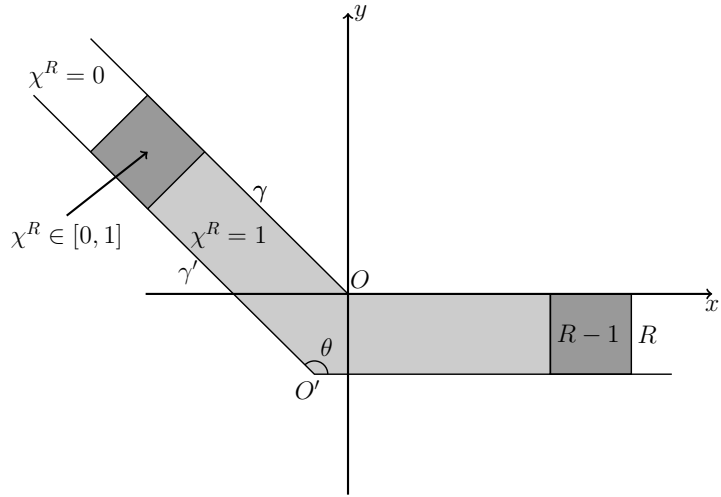}
    \caption{{\sf L}-shaped waveguide $\omega(\theta)$, truncated waveguides  $\omega^{R-1}(\theta)$ и $\omega^R(\theta)$, the support of the function $\chi^R$}
    \label{fig-05}
\end{figure}

Moreover, let us define the truncated waveguide $\omega^R(\theta)$ as a bounded part of $\omega(\theta)$ with the outlets of length $R$ (see. Fig. \ref{fig-05}). The lengths of the outlets of the truncated waveguide $\omega^R$ are measured from the vertex $O$ of the interior angle $\gamma(\theta)$ to the cross-sections perpendicular to the outlets axis.

The following properties of the Dirichlet spectral problem for the Laplace operator in the domains $\omega(\theta)$ are valid.

\begin{proposition} \label{WaveguidePropertiesProp}
\begin{itemize}
\item[(i)] For all $\theta\in (0,\pi)$ the essential spectrum coincides with a ray $[\pi^2,+\infty)$. 

\item[(ii)] For all $\theta\in (0,\pi)$ the discrete spectrum is nonempty. 

\item[(iii)] For all $\theta\in (2 \arctan(\sqrt 3 / 4), \pi)\supset (\pi/2, \pi)$ there is unique eigenvalue $\lambda_1(\omega(\theta))$ below the threshold of the continuous spectrum.

\item[(iv)] 
For all  $\theta\in(0,\pi)$ there exists such $C_\theta>0$ that for all $R>1$ for the first eigenvalue $\lambda_1(\omega^R(\theta))$ of the Laplacian in the truncated waveguide $\omega^R(\theta)$ with the Dirichlet conditions on the boundary $\partial\omega^R(\theta)\cap \partial\omega(\theta)$ and the Neumann conditions on the cross-sections $\partial\omega^R(\theta)\setminus \partial\omega(\theta)$ the following estimate holds 
\[
\lambda_1(\omega(\theta))-C_\theta e^{-2\nu R}\leq
\lambda_1(\omega^R(\theta))\leq
\lambda_1(\omega(\theta)),
\]
with some positive $\nu$.
Moreover, the corresponding eigenfunction $v_1(\omega(\theta))$
decays exponentially along the outlet axes and the following estimate holds 
\[
\|v_1(\omega(\theta));\, H^1(\omega(\theta)\setminus\omega^R(\theta))\| \leq C_\theta e^{-\nu R/2}.
\]

\item[(v)] 
For any integer $k$ there exists a number $\theta_k > 0$ such that when $\theta < \theta_k$, there are at least $k$ eigenvalues below the threshold.

\item[(vi)] 
The first eigenvalue $\lambda_1(\omega(\theta))$ is a monotonically increasing function on $(0,\pi)$, while $\lambda_1(\omega(\theta))\to \pi^2$ as $\theta\to \pi$ and $\lambda_1(\omega(\theta)) \to \pi^2/4$ for $\theta\to 0$.
\end{itemize}
\end{proposition}

The monograph \cite{ExKo15} contains a proof of items (i)-(iii) without exact estimates for~$\theta$. Estimates of the angle at which the eigenvalue is unique are obtained in \cite{Pa17}. The proof of item (iv) for rectangular waveguides is contained in \cite{BaMaNa} and can be easily adapted to the case of an arbitrary angular waveguide, see the note below. Item (v) of Proposition \ref{WaveguidePropertiesProp} and the first statement in item (vi) are proven in \cite{AvBeGiMa91}. The proof of the second statement of the last item of the proposition is given in \cite{DaRa12}.

\begin{remark}
\emph{Proof of item (iv) for an arbitrary angular waveguide.} 
The main difference from the proof presented in \cite{BaMaNa} is that below the threshold the discrete spectrum of an ${\sf L}$-shaped waveguide with a sufficiently sharp angle may consist of several eigenvalues (see item (v) of Proposition \ref{WaveguidePropertiesProp}):
\[
\lambda_1(\omega(\theta)) < \lambda_2(\omega(\theta))\leq\ldots\leq\lambda_\kappa(\omega(\theta)) \leq \pi^2=\lambda_\dagger.
\]
For the first eigenvalues the identity holds
\begin{equation*}
    \lambda_1(\omega(\theta)) - \lambda_1^R(\omega(\theta)) = 
    \frac{(v_1^R, [\Delta, \chi^R]v_1)_{\omega^R(\theta)}}{(v_1^R, \chi^R v_1)_{\omega^R(\theta)}},
\end{equation*}
where $v_1$, $v_1^R$ are the normed eigenfunctions corresponding to  $\lambda_1(\omega(\theta))$ and $\lambda_1^R(\omega(\theta))$, respectively, and $[\Delta, \chi^R]$ is the commutator of the Laplacian and the smooth cut-off function, which equals to 1 in $\omega^{R-1}(\theta)$ and equals to 0 in $\omega(\theta)\setminus\omega^R(\theta)$ (see Fig. \ref{fig-05}).
The method of separating variables guarantees exponential decay (with an exponent of $\sqrt{\pi^2-\lambda_1(\omega(\theta))}$) of the eigenfunction $v_1$ together with its gradient in the outlet of the infinite waveguide, so the numerator of the fraction in the right part is exponentially small, and it remains to show that the denominator is separated away from zero.

The denominator $A_R=(v_1^R, \chi^Rv_1)_{\omega^R(\theta)}$ equal to the coefficient in the decom\-position of $\chi^Rv_1$ into the orthogonal sum 
\begin{equation*}
    \chi^Rv_1 = A_Rv_1^R + B_Rw^R,
\end{equation*}
where $\|w^R; L_2(\omega^R(\theta))\|=1$ and $(v_1^R, w^R)_{\omega^R(\theta)}=0$.
The exponential decay of the eigenfunction and its gradient in the infinite waveguide implies the estimates
\begin{equation} \label{norm_chi_v_1}
    |A_R|^2 + |B_R|^2 = \|\chi_Rv_1; L_2(\omega^R(\theta)\|^2 = 1 - O(e^{-2\sqrt{\pi^2-\lambda_1(\omega(\theta))} R})
\end{equation}
and
\begin{equation} \label{nabla_v_1}
    \|\nabla v_1; L_2(\omega^R(\theta))\|^2=\lambda_1(\omega(\theta)) + O(e^{-2\sqrt{\pi^2-\lambda_1(\omega(\theta))} R}).
\end{equation}
Applying the max-min principle (see \cite[Theorem 10.2.2]{BiSo80}) to the function $w^R$ we obtain 
\begin{equation} \label{nabla_chi_v_1}
    \|\nabla(\chi^Rv_1); L_2(\omega^R(\theta))\|^2\geq|A_R|^2\lambda_1(\omega^R(\theta)) + |B_R|^2\lambda_2(\omega^R(\theta)).
\end{equation}
According to \cite[Lemma 2.1]{BaNa21}, the functions $\lambda_1(\omega^R(\theta))$ and $\lambda_2(\omega^R(\theta))$ monotonically increase with $R$. From the proof in \cite[Proposition 2]{BaNa21} it follows that $\lambda_1(\omega^R(\theta)) \to \lambda_1(\omega(\theta))$ as $R\to+\infty$, and the limit of $\lambda_2(\omega^R(\theta))$ is either equals  $\lambda_2(\omega(\theta))$, or coincides with the threshold $\lambda_\dagger$, or lies above the threshold. The first eigenvalue $\lambda_1(\omega(\theta))$ of the Dirichlet problem is simple, so in any of the above cases the limits of the first two eigenvalues of the truncated waveguides cannot coincide. Thus, from the equalities~\eqref{norm_chi_v_1} and \eqref{nabla_v_1} and inequality~\eqref{nabla_chi_v_1} it follows that $|A_R|$ is equal to 1 up to an exponentially small correction, which proves that the denominator does not vanish for sufficiently large values of $R$. See \cite{BaMaNa} for details.
\end{remark}

\section{Structure of the spectrum} \label{sec-4}

\begin{proposition} \label{Cont_Sp}
The following statement holds
$$
\sigma_{ess}(\mathfrak{A}_\Pi)\supset [\Lambda_\dagger, +\infty), 
$$
where $\Lambda_\dagger$ coincides with the first eigenvalue $\lambda_1(\omega(\beta_\dagger))$ of the Dirichlet spectral problem in a planar broken waveguide $\omega(\beta_\dagger)$ with the angle $\beta_\dagger=\min\limits_{1\leq j\leq n} \beta_j$. 
\end{proposition}

\begin{proof}
For all $\kappa \geq 0$ we construct the Weil sequence for the operator $\mathfrak{A}$ at the point $\lambda_1(\omega(\beta_\dagger)) + \kappa^2$. Denote $v_\dagger$ the non-negative eigenfunction in the waveguide  
$\omega(\beta_\dagger)$, corresponding to the eigenvalue $\lambda_1(\omega(\beta_\dagger))$ and normed in the space $L_2(\omega(\beta_\dagger))$. The layer $\Pi$ for all  $R>0$ contains arbitrarily long cylinders with sections $\omega^R(\beta_\dagger)$. We define the elements of the Weyl sequence in such cylinders by multiplying the eigenfunction $v_\dagger$ in their cross section by an appropriate imaginary exponent depending on the longitudinal coordinate, as well as by the necessary cutting functions. Let us describe the construction in more details.

Suppose $\beta_j=\beta_\dagger$. Consider a part of the layer $\varpi_j=\varpi_\dagger$ from the partition \eqref{partition}.
Associate with $\varpi_\dagger$ the Cartesian system $(x_\dagger, y_\dagger, z_\dagger)=(x_j,y_j,z_j)$. Its sections orthogonal to the $Oz_\dagger$ axis at $z_\dagger>0$ are truncated waveguides of unit width broken at an angle~$\beta_\dagger$. By analogy, we denote the corresponding coordinate system in these sections by $z_\dagger^\perp = (x_\dagger, y_\dagger)$.

Denote by $\mathcal{X}=\mathcal{X}(z)$ a smooth function that vanishes at $z\leq0$, equals one at $z\geq 1$, and satisfies the inequality $0\leq\mathcal{X }(z)\leq 1$. With its help we define cut-off functions along the $z$ axis:
$\mathcal{X}_n(z) = \mathcal{X}(z-2^n)\mathcal{X}(2^{n+1} - z)$.
In each section orthogonal to the $Oz_\dagger$ axis we define, as before, the cut-off function $\chi^{R(n)}=\chi^{R(n)}(x_\dagger, y_\dagger)$, where $R(n)$ is the length of the smallest of the outlets in a cross-section corresponding to the plane $z_\dagger = 2^n$. Note that $R(n)$ tends to infinity as $n$ grows. Finally, we define the sequence of cut-off functions
$\eta_n(x_\dagger,y_\dagger,z_\dagger)=\mathcal{X}_n(z_\dagger)\chi^{R(n)}(x_\dagger,y_\dagger)$.
The sequence of functions
\[
\Psi_n(x)=
\begin{cases}
2^{-n/2}v_\dagger(x_\dag,y_\dag)e^{i\kappa z_\dag}\eta_n(x_\dag,y_\dag,z_\dag),\qquad&(x_\dag,y_\dag,z_\dag)\in\varpi_\dag, \\
0,&(x_\dag,y_\dag,z_\dag)\in\Pi\setminus\varpi_\dagger
\end{cases}
\]
converges weakly to zero in $L_2(\varpi_\dagger)$, since their $L_2(\varpi_\dagger)$-norms do not exceed one and their supports do not intersect. Moreover, the Proposition \ref{WaveguidePropertiesProp} (iv) entails the estimate
\[
\|\Psi_n;\,L_2(\varpi_\dagger)\|^2\geq  1 - C_\dagger e^{-c_\dagger n},\qquad C_\dagger,\,c_\dagger > 0,
\]
from which it follows that the norms of $\Psi_n$ are separated away from zero.
The norm of discrepancy $$\|-\Delta \Psi_n - (\lambda_1(\omega(\beta_\dagger))+\kappa^2)\Psi_n;\,L_2(\varpi_\dagger)\|$$
can be estimated by the sum 
\begin{multline*}
2^{-n/2}  \left( \|\Delta_{z_\dagger^\perp}\eta_n v_\dagger;\,L_2(\varpi_\dagger)\|  +   
2\|\nabla_{z_\dagger^\perp}\eta_n \nabla_{z_\dagger^\perp} v_\dagger;\,L_2(\varpi_\dagger)\|\right)  + \\ + 2^{-n/2}\left( 
\|\partial^2_{z_\dagger}\eta_n v_\dagger;\,L_2(\varpi_\dagger)\| +  
2\kappa\|\partial_{z_\dagger} \eta_n  v_\dagger;\,L_2(\varpi_\dagger)\| \right).
\end{multline*}
Since in each section the supports of the functions $\Delta_{z_\dagger^\perp}\eta_n$ and $\nabla_{z_\dagger^\perp}\eta_n$ are enclosed inside $\omega^{R(n)}(\theta_ \dagger)\setminus\omega^{R(n)-1}(\theta_\dagger)$, while the $v_\dagger$ eigenfunction and its gradient decay exponentially along each outlet (see Proposition \ref{WaveguidePropertiesProp} (iv)), the first term decreases exponentially as $n$ increases.
The supports of the functions $\partial_{z_\dagger} \mathcal{X}_n$ and $\partial^2_{z_\dagger} \mathcal{X}_n$ contained in the union of two cylinders of the unit height:
$$\omega^{R(n)}(\theta_\dagger)\times\Big([2^n, 2^n+1]\cup [2^{n+1}-1, 2^{n+1}]\Big).$$
The function $v_\dagger$ is normalized in the space $L_2(\omega(\theta_\dagger))$, therefore, due to the exponential smallness of the normalizing factor $2^{-n/2}$, the second term in the discrepancy estimate also decreases exponentially with increasing of $n$. As a result, the $L_2(\varpi_\dagger)$-norm of the discrepancy tends to zero as $n$ increases.

Thus, the functions $\Psi_n$ form the Weil sequence for the number $\lambda_1(\omega(\beta_\dag))+\kappa^2$, which according to \cite[Theorem 9.1.2]{BiSo80} belongs to the essential spectrum.
\end{proof}

To prove the finiteness of the discrete spectrum, we need an auxiliary Hardy-type inequality.
\begin{lemma} 
For all functions $v\in H^1(1,+\infty)$ the inequality holds 
\[
\|\rho v;L_2(2,+\infty)\|^2\leq
4\|v';L_2(2,+\infty)\|^2+2\|v';L_2(1,2)\|^2+
2\|v,L_2(1,2)\|^2
\]
for $\rho(z)=z^{-1}$.
\end{lemma}
\begin{proof}
Denote the left-hand side of the inequality by $A^2$ and, using the Newton--Leibniz formula, estimate it as follows
\begin{multline*}
A^2=
\int_2^{+\infty}z^{-2}\left|2 \int_2^z v(t)v'(t)dt + v^2(2)\right|dz\leq \\
\leq
2\int_2^{+\infty}z^{-2}\int_2^z|v(t)||v'(t)|dt\,dz +
\frac{1}{2}|v(2)|^2.
\end{multline*}
Changing the order of integration using the Tonelli theorem in the last integral and then applying the Cauchy--Swartz inequality, we conclude that
\[
A^2\leq
2AB +
\frac12 C^2,\quad  \mbox{где} \quad
B = \left(\int_2^{+\infty}|v'(t)|^2\right)^{\frac12},\, C = |v(2)|.
\]
Applying the arithmetic -- geometric mean inequality to the right hand side, we arrive at 
\[
A^2\leq \frac12A^2 + 2B^2 + \frac12 C^2,
\]
and hence
\begin{equation} \label{Almost_Hardy_enq} 
A^2\leq 4 B^2 + C^2.
\end{equation}
To estimate the last term $C^2=|v(2)|^2$, we write the Newton--Leibniz formula
\[
v(t)-v(s)=\int_s^tv'(z)\,dz
\]
for $t, s \in [1, 2]$. Integrating both parts of this equality with respect to the variable $s$ from 1 to 2, we obtain the relation
\begin{multline*}
v(t)=\int_1^2v(t)\,ds\leq\int_1^2|v(s)|\,ds+\int_1^2\int_s^t|v'(z)|\,dz\,ds\\
\leq\int_1^2|v(s)|\,ds+\int_1^2|v'(z)|\,dz\leq \|v;L_2(1,2)\|+\|v';L_2(1,2)\|.
\end{multline*}
In the last estimate, we used the Cauchy--Swartz inequality. As a result, we get that
\begin{equation} 
\label{pointwise_estimate}
|v(t)|^2\leq 2\|v';L_2(1,2)\|^2+2\|v;L_2(1,2)\|^2
\end{equation}
for $t\in[1,2]$. Combining together the inequalities \eqref{Almost_Hardy_enq} and \eqref{pointwise_estimate},
we arrive at the required assertion.
\end{proof}

We need a slightly more rough estimate from the following corollary.
\begin{corollary}
\label{Lem_almost_Hardy_enq_cor}
For any $v\in H^1(1,+\infty)$ and any $R_0\geq 2$ the inequality
\[
\|\rho v;L_2(R_0,+\infty)\|^2\leq
4\|v';L_2(1,+\infty)\|^2+
2\|v,L_2(1,R_0)\|^2
\]
for $\rho(z)=z^{-1}$.   
\end{corollary}

\begin{theorem} \label{Finitness_of_discrete_sp}
\label{ess_sp_th}
The discrete spectrum of the operator $\mathfrak{A}_\Pi$ in the angle layer $\Pi$ is finite and lies below the threshold $\Lambda_\dagger$.
The essential spectrum of the operator $\mathfrak{A}_\Pi$ occupies the ray $[\Lambda_\dagger,+\infty)$.
\end{theorem}

\begin{proof} 
From the Proposition \ref{Cont_Sp} it follows that the discrete spectrum lies below the threshold. To prove its finiteness, we will use an analogue of the Dirichlet--Neumann bracketing (see \cite[XIII.15]{RS}), namely, cut the Dirichlet layer $\Pi$ into domains $\varpi_j$ (partition \eqref{partition}), preserving the Dirichlet condition on the boundary $\partial \Pi\cap \partial \varpi_j$ and assigning the Neumann condition on the flat boundaries of the cuts $\pi_j\cap \Pi$. This partition corresponds to the operator $\widetilde{\mathfrak{A}}_\Pi$ with quadratic form
$$
\widetilde{\mathfrak{a}}_\Pi(u,u)=\sum_{j=1}^n \|\nabla u; L_2(\varpi_j)\|^2,
$$
defined on the space $$
\mathcal{H}(\Pi) = 
\{u\in L^2(\Pi) \colon u|_{\varpi_j}\in H^1(\varpi_j) \mbox{ for } 1\leq j\leq n,  u|_{\partial \Pi}=0 \}.$$

According to the max-min principle (see 
 \cite[Theorem  XIII.2]{RS} or \cite[Theorem 10.2.2]{BiSo80}), the eigenvalues of the operator~$\widetilde{\mathfrak{A}}_\Pi$ do not exceed the corresponding eigenvalues of the operator~$\mathfrak{A}_\Pi$, and hence the finiteness of the discrete spectrum of the mixed boundary value problem entails the finiteness of the discrete spectrum of the operator~$\mathfrak{A}_\Pi$.

The operator $\widetilde{\mathfrak{A}}_\Pi$ is a direct sum of Laplacians with mixed boundary conditions on the partition elements, so it suffices to prove that, given on $\varpi_j$, the Laplace operator  with the Neumann conditions on the cuts $\partial \varpi_j\setminus (\Gamma\cup\Gamma ')$ and the Dirichlet conditions on the rest of the boundary has a finite number of eigenvalues below $\Lambda_\dagger$.

In what follows we will not write the index $j$. The assertion that the discrete spectrum is finite follows from the existence of such finite-dimensional space 
$E\subset L_2(\Pi)$, that for all $u$ orthogonal $E$ the following inequality holds

\begin{equation} \label{Cut_off_est_from_below}
    \frac{\|\nabla u;\ L_2(\varpi)\|^2}{\|u;\ L_2(\varpi)\|^2}\geq\Lambda_\dagger.
\end{equation}

\begin{figure}[ht]
    \includegraphics[width=0.5\textwidth]{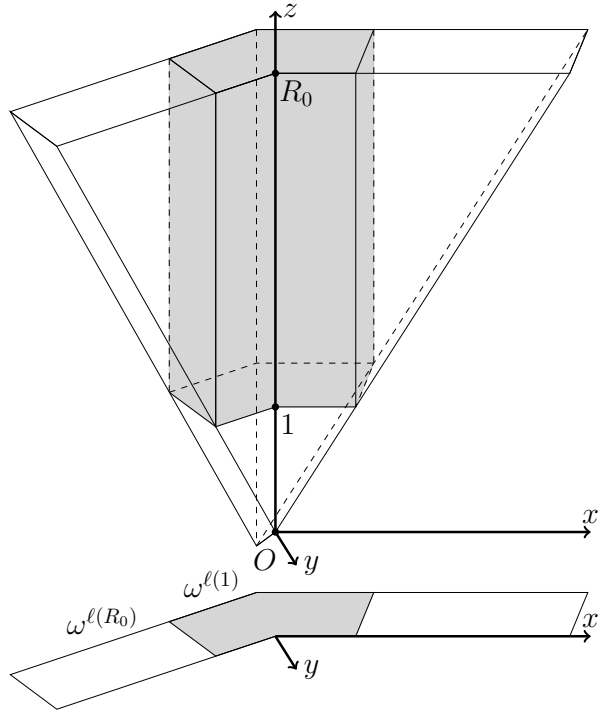}
    \caption{The set $\varpi_\bullet$ and the part of the cylinder $Q$ that fell into it in the case of a symmetric domain $\varpi$}
\end{figure}

Let us introduce some notation. We denote by $\ell(R)=c R$ the length of the smaller of the two arms in the cross-section $\varpi$ by the plane $z=R$. Denote by~$Q$ the cylinder $\omega^{\ell(1)}\times(1,+\infty)$ inscribed in $\varpi$. The set $\varpi$ is divided into two parts: the compact $\varpi_\bullet=\varpi\cap\{z\leq R_0\}$ and the unbounded $\widetilde{\varpi}=\varpi\cap\{z\geq R_0\}$ (parameter $R_0$ will be chosen later). For the function 
$u\in\mathcal{H}(\Pi)$
denote by $u_z(x,y) =u(x,y,z)$.

For $z\geq 0$ the inequality holds
\begin{equation} \label{ReyleighQuotTruncWG}
\|\nabla u_z;L_2(\omega^{\ell(z)})\|^2\geq \lambda_1(\omega^{\ell(z)})\|u_z;L_2(\omega^{\ell(z)})\|^2,
\end{equation}
and for $z\geq 1$ outside the cylinder $Q$, due to the Friedrichs inequality, we can write that 
\begin{equation} \label{ReyleighQuotCoreWG}
\|\nabla u_z;L_2(\omega^{\ell(z)})\|^2\geq\|\nabla u_z;L_2(\omega^{\ell(z)}\setminus\omega^{\ell(1)})\|^2\geq
\pi^2\|u_z;L_2(\omega^{\ell(z)}\setminus\omega^{\ell(1)})\|^2.
\end{equation}
Multiplying the inequality \eqref{ReyleighQuotTruncWG} by 
$(1-(4z\pi)^{-2})$
\[
(1-(4z\pi)^{-2})\|\nabla u_z;L_2(\omega^{\ell(z)})\|^2
\geq(1-(4z\pi)^{-2})\lambda_1(\omega^{\ell(z)})\|u_z;L_2(\omega^{\ell(z)})\|^2,
\]
and writing the inequality \eqref{ReyleighQuotCoreWG} as
\[
(4z\pi)^{-2}\|\nabla u_z;L_2(\omega^{\ell(z)}\setminus\omega^{\ell(1)})\|^2+
(4z)^{-2}\|u_z;L_2(\omega^{\ell(1)})\|^2\geq(4z)^{-2}\|u_z;L_2(\omega^{\ell(z)})\|^2,
\]
add the resulting estimates.
As a result, we have
\begin{multline} \label{MainSpatialEst}
\|\nabla u_z;L_2(\omega^{\ell(z)})\|^2+
(4z)^{-2}\|u_z;L_2(\omega^{\ell(1)})\|^2\geq\\ \geq
\Big((1-(4z\pi)^{-2})\lambda_1(\omega^{\ell(z)})+(4z)^{-2}\Big)\|u_z;L_2(\omega^{\ell(z)})\|^2.
\end{multline}
By virtue of Proposition \ref{WaveguidePropertiesProp} (iv) the coefficient on the right-hand side of \eqref{MainSpatialEst} satisfies the inequality
\[
(1-(4z\pi)^{-2})\lambda_1(\omega^{\ell(z)})+(4z)^{-2}\geq
\lambda_1(\omega)+
(4z)^{-2}(1-\lambda_1(\omega)\pi^{-2})-C e^{-\nu \ell(z)},
\]
and hence there exists $R_0\geq 2$ such that for $z\geq R_0$ the inequality holds
$$(1-(4z\pi)^{-2})\lambda_1(\omega^{\ell(z)})+(4z)^{-2} > \lambda_1(\omega),$$ 
and in the end
\begin{equation*}
\|\nabla u_z;L_2(\omega^{\ell(z)})\|^2+
(4z)^{-2}\|u_z;L_2(\omega^{\ell(1)})\|^2\geq
\lambda_1(\omega) \|u_z;L_2(\omega^{\ell(z)})\|^2.
\end{equation*}

Integrating the last estimate over $z\in[R_0,+\infty)$, we derive the inequality
\[
\|\nabla_{z^\perp} u; L_2(\widetilde \varpi)\|^2+
\frac{1}{16}\|\rho u;L_2(\widetilde \varpi\cap Q)\|^2\geq
\lambda_1(\omega)\|u;L_2(\widetilde \varpi)\|^2.
\]
Recall that, as before, $\rho(z)=z^{-1}$.
Applying the Corollary \ref{Lem_almost_Hardy_enq_cor} to the second term, we obtain the relation
\begin{multline} \label{LemmaEnq1}
\|\nabla_{z^\perp} u; L_2(\widetilde \varpi)\|^2+
\frac{1}{4}\|\partial_z u;L_2( Q)\|^2\geq\\
\geq\lambda_1(\omega)\|u;L_2(\widetilde \varpi)\|^2-
\frac18\|u;L_2(\varpi_\bullet\cap Q)\|^2.
\end{multline}
Note that by imposing the orthogonality conditions to the first eigenfunctions 
of the Laplace operator with the Dirichlet conditions on $\partial\varpi_\bullet\cap\partial\Pi$ and the Neumann conditions on the rest of the boundary
on the function $u$, for any $M>0$ one can achieve the inequality
\[
\|\nabla u; L_2(\varpi_\bullet)\|^2>M\|u;L_2(\varpi_\bullet)\|^2.
\]
Adding this inequality, divided by 2, with the inequality \eqref{LemmaEnq1}, we obtain the estimate
\begin{multline*}
\|\nabla_{z^\perp} u; L_2(\widetilde \varpi)\|^2+
\frac{1}{4}\|\partial_z u;L_2(\widetilde \varpi\cap Q)\|^2+\\
\frac14\|\partial_z u;L_2(\varpi_\bullet\cap Q)\|^2
+\frac12\|\nabla u;L_2(\varpi_\bullet)\|^2\geq\\
\lambda_1(\omega)\|u;L_2(\widetilde \varpi)\|^2-
\frac18\|u;L_2(\varpi_\bullet\cap Q)\|^2+
\frac M2\|u;L_2(\varpi_\bullet)\|^2.
\end{multline*}
If 
$M\geq2(\lambda_1(\omega)+1)$,
then
\[
\|\nabla u; L_2(\varpi)\|^2\geq\lambda_1(\omega)\|u;L_2(\varpi)\|^2,
\]
which gives an estimate for the threshold of the continuous spectrum.
One can choose the linear span of eigenfunctions from the orthogonality condition on $\varpi_\bullet$, extended by zero to $\varpi$, as the space $E$.

It follows from the Proposition \ref{Cont_Sp} that the ray $[\Lambda_\dagger,+\infty)$ is contained in the essential spectrum.
According to the max-min principle, 
the inequality \eqref{Cut_off_est_from_below}
holds true for all functions from  $\lefteqn{\overset\circ{\hphantom{H'}\vphantom{\prime}}}H\vphantom{H}^1(\Pi)$
that are orthogonal to the subspace $E$ in $L_2(\Pi)$ and hence it guarantees that $\Lambda_\dagger$ is a cutt-off point of the essential spectrum, and all eigenvalues below this point belong to the discrete spectrum. 
\end{proof}

\section{An example of the absence of a discrete spectrum} 
\label{sec-6}
Despite the fact that all the layers constructed by trihedral angles are structurally similar, it turns out that the discrete spectrum does not appear in all cases. A corresponding example is given in the following theorem.

\begin{theorem} \label{abs_of_ds}
In the layer $\Pi$, constructed from a trihedral angle with two right vertex angles at the vertex $O'$ (see Fig. \ref{fig-03} on the left), for sufficiently small values of the third vertex angle $\alpha$ the spectral Dirichlet problem for the Laplace operator has no eigenvalues below the threshold of the essential spectrum.
\end{theorem}

\begin{figure}[ht!]
    \centering
    \includegraphics[width=0.42\textwidth]{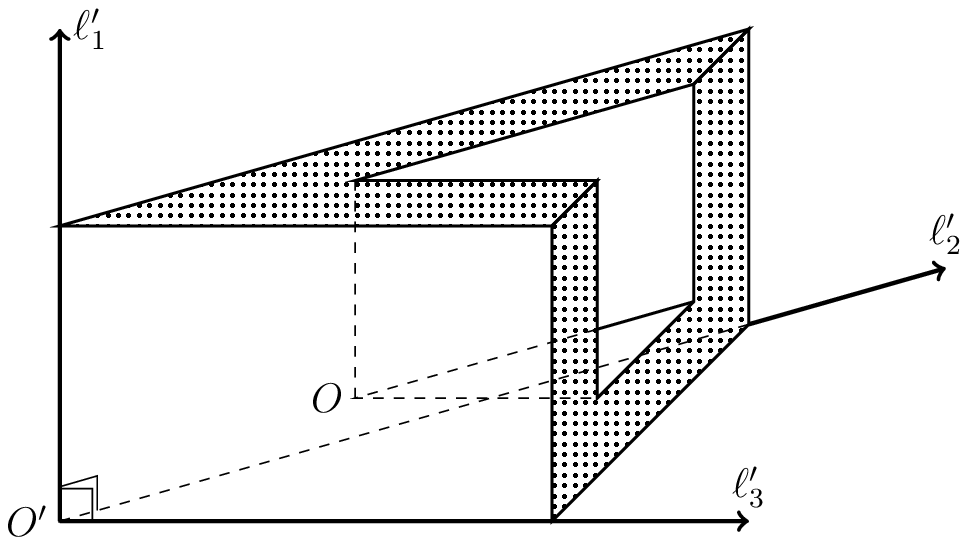}\qquad
    \includegraphics[width=0.42\textwidth]{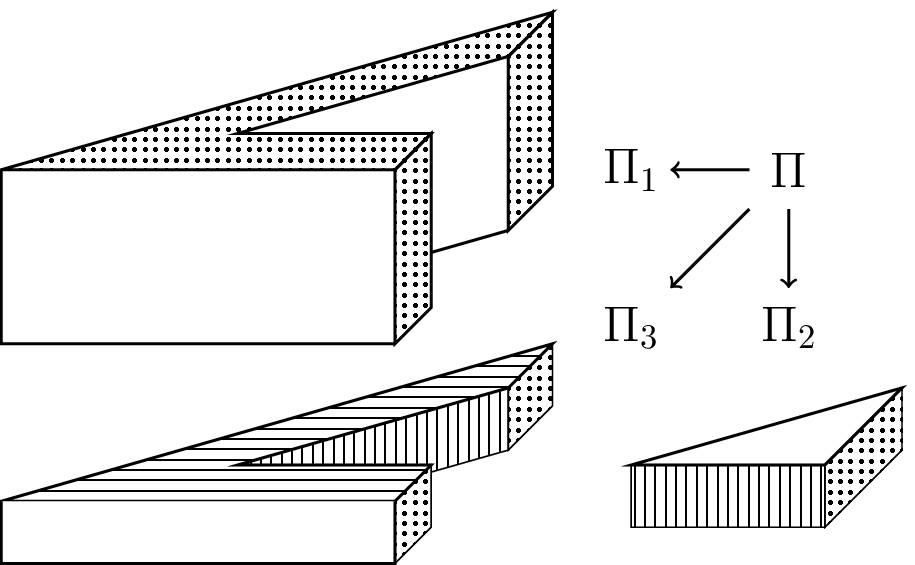}
    \caption{A layer constructed from a trihedral angle with two right vertex angles and a small third vertex angle and its partition}
    \label{fig-03}
\end{figure}

\begin{proof}
The essential spectrum in the domain $\Pi$ by the Theorem \ref{ess_sp_th} is determined by the smallest dihedral angle, which is equal to $\alpha$. Thus, it is sufficient to check that for any function $u\in \lefteqn{\overset\circ{\hphantom{H'}\vphantom{\prime}}}H\vphantom{H}^1(\Pi)$ the following inequality holds
\begin{equation}
\label{absent_of_ev_eq}
\|\nabla u; L_2(\Pi)\|^2\geq \lambda_1(\omega(\alpha)) \|u; L_2(\Pi)\|^2.
\end{equation}
Assume that $\ell_1$ is an edge of a smallest dihedral angle equal to $\alpha$. Let's use the coordinate system $(x_1, y_1, z_1)$ associated with this angle, as described in \S \ref{sec-2}.
To check the inequality \eqref{absent_of_ev_eq} we divide $\Pi$ into three parts $\Pi= \Pi_1\cup \Pi_2\cup \Pi_3$
\begin{eqnarray*}
&&\Pi_1=\Pi \cap \{z_1>0\},    \\
&&\Pi_2= \Pi\cap \{0< y_1 < \tan(\alpha) x_1 \},\\
&&\Pi_3= \Pi\setminus (\Pi_1\cup \Pi_2).
\end{eqnarray*}
The parts $\Pi_j$ are shown schematically in Fig. \ref{fig-03} on the right.

In the part $\Pi_1$ we integrate over $z_1$ the Friedrichs inequality on the cross-sections perpendicular to the ray $\ell_1$ and obtain the estimate
$$
\|\nabla u; L_2(\Pi_1)\|^2\geq \lambda_1(\omega(\alpha)) \|u;L_2(\Pi_1)\|^2.
$$
In the part $\Pi_2$ the Friedrichs inequality on segments parallel to the ray $\ell_1$ leads to the estimate
$$
\|\partial_{z_1} u; L_2(\Pi_2)\|^2\geq \pi^2 \|u;L_2(\Pi_2)\|^2 \geq \lambda_1(\omega(\alpha)) \|u;L_2(\Pi_2)\|^2.$$
In the part $\Pi_3$ for the $z_1$-derivative we write the Friedrichs inequality along segments parallel to the $\ell_1$ axis
\begin{equation}
    \label{partial_z_1_estimate}
    \|\partial_{z_1} u; L_2(\Pi_3)\|^2\geq \frac{\pi^2}{4}\|u; L_2(\Pi_3)\|^2.
\end{equation}
To estimate the transverse gradient, we split the part $\Pi_3$ by the vertical plain of symmetry contained $O'\ell'_1$, and write  the Friedrichs inequality along segments, that are orthogonal to the outer boundary.
As the lengths of the segments do not exceed one, we obtain that
\begin{equation}
    \label{nabla_z_1perp_estimate}
    \|\nabla_{z_1^\perp} u; L_2(\Pi_3)\|^2\geq \frac{\pi^2}{4}\|u; L_2(\Pi_3)\|^2.
\end{equation}
Combining together the inequalities  \eqref{partial_z_1_estimate} and \eqref{nabla_z_1perp_estimate}, we conclude that
$$
\|\nabla u; L_2(\Pi_3)\|^2\geq (\pi^2/4+\pi^2/4) \|u;L_2(\Pi_3)\|^2 \geq \lambda_1(\omega(\alpha)) \|u;L_2(\Pi_3)\|^2.
$$
The last inequality by point (vi) of Proposition \ref{WaveguidePropertiesProp} is true for sufficiently small angles~$\alpha$.
\end{proof}

\section{Existence of a discrete spectrum for regular angles}
\label{sec-7}

Recall that we call an angle layer regular if its outer boundary $\Gamma'$ forms a regular polyhedral cone, i.e. all dihedral angles $\beta_j$ are equal and all vertex angles $\alpha_j$ at the vertex $O'$ are equal. In what follows we denote the corresponding quantities by $\beta$ and~$\alpha$; the plane waveguide $\omega(\beta)$ in the section perpendicular to the edge of the angle layer is denoted by $\omega$.

\begin{theorem} \label{ExistDiscrSp}
In a regular angular layer the discrete spectrum is not empty.
\end{theorem}
\begin{proof}

According to the maximum-minimum principle 
to prove the existence of an eigenvalue below the threshold $\Lambda_\dagger = \lambda_1(\omega)$, it is sufficient to construct a function $V\in \lefteqn{\overset\circ{\hphantom{H'}\vphantom{\prime}}}H\vphantom{H}^1(\Pi)$,
such that the following inequality holds
\begin{equation} \label{FirstEVEnq}
\|\nabla V;\, L_2(\Pi)\|^2 - \lambda_1(\omega)\|V;\, L_2(\Pi)\|^2
< 0.
\end{equation}
Similarly to the proof of the Theorem \ref{Finitness_of_discrete_sp}, we cut the layer into parts $\varpi_j$ (partition~\eqref{partition}). Since they are all the same, we will denote one of them as $\varpi_\dagger$.
Define in $\varpi_\dagger$ the function $V^\varepsilon(x_\dagger, y_\dagger, z_\dagger) = v_\dagger(x_\dagger, y_\dagger) e^{-\varepsilon z_\dagger}$,
where  $\varepsilon > 0$, a $v_\dagger$ is the first eigenfunction of the Dirichlet Laplacian in the broken waveguide $\omega$.
A regular angular layer has a rotational symmetry, thus the function~$V^\varepsilon$ is continuous  on the faces $\pi_{j-1}\cap\varpi_j$ and $\pi_j\cap\varpi_j$, and hence $V^\varepsilon\in\lefteqn{\overset\circ{\hphantom{H'}\vphantom{\prime}}}H\vphantom{H}^1(\Pi)$.
Substituting $V^\varepsilon$ into the left-hand side of the inequality~\eqref{FirstEVEnq}, after integrating by parts we have
\begin{multline} \label{VinFeirstEnq}
(-\Delta_{z_\dagger^\perp} V^\varepsilon, V^\varepsilon)_{\varpi_\dagger} - \lambda_1(\omega)\|V^\varepsilon;\, L_2(\varpi_\dagger)\|^2 +\\+
\varepsilon^2\|V^\varepsilon;\,L_2(\varpi_\dagger)\|^2 + \int_{\partial \varpi_\dagger} \vec{n} \cdot \nabla_{z_\dagger^\perp} v_\dagger(z^\perp_\dagger) v_\dagger(z^\perp_\dagger) e^{-2\varepsilon z_\dagger} dS,
\end{multline}
where $\vec{n} = (n_x, n_y, n_z)$ is the outer normal vector to the boundary  $\partial \varpi_\dagger \setminus (\Gamma \cup \Gamma')$, and $dS$~is the standard surface Lebesgue measure on this boundary. 
Since $v_\dagger$ is an eigenfunction corresponding to $\lambda_1(\omega)$, the first two terms cancel.
The plane of symmetry passing through the $Oz_\dagger$ axis divides the layer $\varpi_\dagger$ into two parts $\varpi_+$ and $\varpi_-$, we assume that $\varpi_+$ is directed along $Ox_\dagger$.
Denote the common end of $\varpi_\dagger$ and $\varpi_+$ lying in the plane $z_\dagger = \cot(\alpha / 2) x_\dagger$ by~$\pi_+$.
As a result, the sum \eqref{VinFeirstEnq} is converted to the form
\[
\varepsilon^2\|V^\varepsilon;\,L_2(\varpi_\dagger)\|^2 + 2\!\int_{\pi_+}
\!\!\!\!n_x\partial_{x_\dagger} v_\dagger(x_\dagger, y_\dagger) v_\dagger(x_\dagger, y_\dagger)
e^{-2\varepsilon \cot(\alpha/2) x_\dagger}
dS(x_\dagger,y_\dagger).
\]
Given that
$dS(x_\dagger, y_\dagger) = \sin^{-1}(\alpha/2)\,dx_\dagger dy_\dagger$,
and $n_x = \cos(\alpha_\dagger / 2)$, the last integral can be written
as a double integral over one of the two symmetrical halves of the waveguide~$\omega$. Let's denote one of these halves directed along the $Ox_\dagger$ axis by $\omega_+$, and the other by~ $\omega_-$.
In this case the expression takes the form
\[
\varepsilon^2\|v_\dagger;\,L_2(\varpi_\dagger)\|^2 + 2 \cot\left(\frac\alpha2\right)\int_{\omega_+}  \partial_{x_\dagger} v_\dagger(x_\dagger, y_\dagger) v_\dagger(x_\dagger, y_\dagger)e^{-2\varepsilon \cot(\alpha/2) x_\dagger} d x_\dagger dy_\dagger.
\]
Integrating by parts we obtain 
\begin{multline} \label{LastEnqNegRayQ}
    \varepsilon^2\|v_\dagger;\,L_2(\varpi_\dagger)\|^2 +
    2\varepsilon \cot^2\left(\frac\alpha2\right)\int_{\omega_+}\!\! |v_\dagger(x_\dagger, y_\dagger)|^2\exp(-2\varepsilon\cot(\alpha/2) x_\dagger) dx_\dagger dy_\dagger - \\ - 2 \cot\left(\frac\alpha 2\right) \sin\left(\frac\beta2\right) \int_{\gamma_0}|v_\dagger(x_\dagger(\tau),y_\dagger(\tau))|^2\exp(-2\varepsilon \cot(\alpha/2) x_\dagger(\tau)) d\tau,
\end{multline}
where $\gamma_0 = \partial\omega_-\cap\partial\omega_+$, and
$\tau$~is the natural parametrization of the boundary along~$\gamma_0$.
Note that for sufficiently small $\varepsilon$ the values \eqref{LastEnqNegRayQ} are negative, so for small $\varepsilon$
the function~$V^\varepsilon$ satisfies the inequality \eqref{FirstEVEnq}.
\end{proof}

\section{Possible generalizations}
\label{sec-8}

The results of Theorems \ref{Finitness_of_discrete_sp} and \ref{ess_sp_th} can be extended to layers constructed from angles that are not necessarily circumscribed around balls. To do this, it is sufficient to construct an analogue of the partition \eqref{partition}, the structure of which will be violated only in a compact neighborhood of the angle vertex, which will not affect the essential spectrum of the operator. Note that the thickness of the walls of the angular layer can also be made different for different faces. In general, the results will remain similar, and only the analysis of the spectra of {\sf L}-shaped waveguides will become more complicated.

For a regular polygonal layer with small planar angles near the vertex (and dihedral angles close to $\pi$) it is easy to prove that the number of eigenvalues below the continuous spectrum can be made arbitrarily large. The proof is absolutely analogous to the results in \cite{DaRa12, BaMaNaZaa}. We expect that there is also monotonicity with respect to the opening angle, similar to one proved in \cite[Proposition 3.1.]{DaLaRa12}, but at the moment this statement remains unproven. In addition, most likely, for sufficiently large opening angles, we expect uniqueness of the eigenvalue below the continuous spectrum.

Increasing the number of faces in a regular angle inscribed in a given cone should also lead, in our opinion, to an increase in the number of eigenvalues below the continuous spectrum, since the region visually becomes similar to a conical layer (cf. works \cite{ExTa10, DaOuRa15}). However, it turns out that it is not true that a regular polygonal layer can be inscribed in a conical layer, so it is not possible to prove this statement by simple estimates between operators.

\end{document}